\documentclass[preprint,12pt,fleqn]{elsarticle}

\usepackage{amssymb,amsfonts,amsmath,amscd,amsthm,latexsym}
\usepackage{graphicx}
\usepackage{mathptmx}      
\usepackage[english]{babel}
\usepackage{geometry}
\usepackage[all]{xy}
\usepackage{appendix}
\usepackage[active]{srcltx}
\usepackage{color}

\newtheorem{thm}{Theorem}[section]
\newtheorem{cor}[thm]{Corollary}
\newtheorem{lem}[thm]{Lemma}
\newtheorem{prop}[thm]{Proposition}
\theoremstyle{definition}
\newtheorem{defn}[thm]{Definition}

\newtheorem{nim}[thm]{}

\newtheorem{set}[thm]{Setup}

\newcommand{\Tr}{\operatorname{Tr}}

\newcommand{\Aut}{\operatorname{Aut}}

\newcommand{\Br}{\operatorname{Br}}

\newcommand{\Sum}{\displaystyle\sum}

\begin{document}

\begin{frontmatter}



\title{Isomorphisms of fusion subcategories on permutation algebras}

\author[UB,UT]{Tiberiu Cocone\c t}
\ead{tiberiu.coconet@math.ubbcluj.ro}
\address[UB]{Faculty of Economics and Business Administration, Babe\c s-Bolyai University, Str. Teodor Mihali, nr.58-60 , 400591 Cluj-Napoca, Romania}
\address[UT]{Department of Mathematics, Technical University of Cluj-Napoca, Str. G. Baritiu 25,  Cluj-Napoca 400027, Romania}
\author[UT]{Constantin-Cosmin Todea \corref{cor1}}
\ead{Constantin.Todea@math.utcluj.ro}

\cortext[cor1]{Corresponding author}

\begin{abstract} We prove two isomorphisms of some subcategories of fusion systems defined on $N$-interior $p$-permutation $G$-algebras by  extending the main results obtained by M. E. Harris  in \cite{Ha}.
\end{abstract}


\end{frontmatter}

\section{Introduction}\label{sec1}
There are published  results  obtained in the case of blocks of  group algebras that are extended to the general situation of $p$-permutation $G$-algebras ($p$ is a prime and $G$ is a finite group), for instance this is the case of the Brauer pairs of $p$-blocks, that were introduced in this general situation in  \cite{BrPu}. The first author introduced the definition of covering points \ref{definition41} on certain $p$-permutation algebras (see Section \ref{sec5}), which  was used to extend the Harris-Kn\"{o}rr correspondence. M. E. Harris obtained in \cite{Ha} two  isomorphisms of categories involving the generalized Brauer pairs  determined by a $G$-invariant block of  $kN,$ where $N$ is normal in $G.$ The main purpose of this paper is to  extend these results to a class of $p$-permutation algebras.  

Throughout the paper $k$ is a field of characteristic $p,$ not necessarily algebraically closed and, $N$ is a normal subgroup of the finite group $G$ such that $p$ divides $|N|.$  Section \ref{sec2} contains a brief discussion  involving saturated triples and states the main result of this section, Proposition \ref{prop22}, which is important  for the following sections. Section \ref{sec3} proves an extension of a useful result, that will be used to show an important ingredient (Corollary \ref{corollary34}) for the proof our first main result. Section \ref{sec4} contains the first main result of the paper, that is Theorem \ref{theorem39}. This theorem is an \textit{extension} of \cite[Theorem 6]{Ha} to a class of $p$-permutation algebras. Sections \ref{sec5} and \ref{sec6} gather information on covering points on $p$-permutation algebras. Finally, in Section \ref{sec7} we prove the last two main results of the paper, namely Theorem \ref{thecorresp} and Theorem \ref{theorem63}. In Theorem \ref{thecorresp} we obtain a bijective correspondence involving covering points of some group acted $p$-permutation algebras
and  Theorem \ref{theorem63} is a generalization of \cite[Corollary 8]{Ha}.

Extending results from block algebras to  $p$-permutation algebras is  still of interest. For example \cite[Part IV, Section 7, Question 5]{AsKeOl} contains  an open question about the realizibility of any saturated fusion system: given a saturated
fusion system $\mathcal{F}$ on a finite $p$-group $D$, are there  a finite group $G$, a
$p$-permutation algebra $A$, and a primitive idempotent $b$ of $A^G$ such that
$(A,b,G)$ is a saturated triple (see Section \ref{sec2}) whose  corresponding fusion system
is isomorphic to $\mathcal{F}$? A different motivation for the results in this paper is the fact that we are able to generalize from blocks of group algebras over finite groups to points of some $G$-algebras, using heavily techniques different than those of Harris in \cite{Ha}, the result \cite[Lemma 8.6.4]{Link}, which is an important ingredient of \cite[Theorem 8.6.3]{Link}.

For basic facts regarding points on $G$-algebras we follow \cite{Th} and for the language of Brauer pairs in the context of $p$-permutation algebras we refer to \cite{AsKeOl} and \cite{BrPu}. All algebras that appear in this paper are assumed to be finite dimensional over $k$. We use  the   lifting idempotents techniques as presented in \cite[ Section 3]{Pu}.   $A^*$ denotes the group of units of a ring  $A$ and we use conjugation on the left:$~^af:=afa^{-1}$ for any $f\in A, a\in A^*.$ Depending on the situation,  we sometimes  let "$\cdot$" denote the usual multiplication between two elements of  a $k$-algebra. Most frequently    we  only use juxtaposition.

We recall the basic facts about points, pointed groups and defect groups attached to pointed groups, for details 
 see \cite[Section 13 and Section 18]{Th}. For any $G$-algebra $A$ and for any subgroup $H$ of $G,$ a point of $A^H$ is the $(A^H)^*$-conjugacy class of a primitive idempotent of $A^H$. 
If $\alpha $ is a point of $A^H,$ that is 
$\alpha=\{aia^{-1}| a\in (A^H)^* \},$ the pair
$(H,\alpha)$   is called a pointed group. We denote this pair by  $H_{\alpha}.$

 A $p$-subgroup $D$ of $H$ is a defect group of $H_{\alpha},$ or a defect group of $\alpha,$ if $\alpha\subseteq A_D^H$ and $\Br^A_D(\alpha)\neq 0,$ here $\Br^A_D$ is the Brauer homomorphism determined by $D.$ If $\alpha\subseteq A^H$ is a point admitting defect group $D$ then  any  $i\in \alpha$ may be regarded as a point of $iA^Hi=(iAi)^H$ consisting of a unique element, namely $\{i\}.$ It is clear that $\{i\}$ has defect group $D.$ In this context we say that $i$ has defect group $D.$ Note that $D$ is a defect group of $\alpha$ if and only if it is a defect group for any $i\in \alpha. $
 
 Any $G$-algebra embedding  $\mathcal{F}:A \rightarrow B$
  induces an injective map between the points of $A^H$ into  the points of $B^H,$ for any subgroup   $H$ of $G.$ This injective map preserves the defect (pointed groups) between  the corresponding points. For details see see \cite[p. 58]{Th} and  \cite[Proposition 15.1 (a), (c) ]{Th}.
We  freely use various characterizations of defect groups, see also \cite[ 18.3 (iv)]{Th}. Even though the main results work only in special cases of interior $G$-algebras, for a finite group $G,$ we \textit{emphasize} that each section has its own set of assumptions representing the most general setting for the results developed there.  

\section{On the definition of saturated triples}\label{sec2}
First  we prove a technical lemma regarding idempotents   in a  $k$-algebra $A,$ which is a consequence of \cite[Proposition 3.19]{Pu}.
\begin{lem}\label{Lemma71} Let $A$ be a unital $k$-algebra. Let $e,f$ be two idempotents of $A$ such that $f\in Z(A).$ The following statements are true.
\begin{itemize}
\item[a)] If $f$ is primitive in $A$ and $fe\neq 0$ then $f e=e f=f;$
\item[b)] If $e$ is primitive in $A$ and $f e\neq 0$ then $fe=e f=e.$
\end{itemize}
\end{lem}
\begin{proof}
Fix a decomposition of $e$ into pairwise orthogonal primitive idempotents $e=\sum_{i\in I}i.$ According to \cite[Proposition 3.19]{Pu} for any idempotent $f$ (central or non-central) there is an $A$-conjugate ${}^af$ such that 
$e\cdot{}^af={}^af \cdot e.$ Further, if $$e\cdot{}^af={}^af \cdot e\neq 0$$ then $$e\cdot{}^af={}^af \cdot e=\sum_{i\in S}i,$$ where $S\subseteq I$ such that 
$$\sum_{i\in S}i=e\cdot \left(\sum_{i\in S}i\right)=\left(\sum_{i\in S}i\right)\cdot e,$$
and
$$\sum_{i\in S}i={}^af\cdot\left(\sum_{i\in S}i\right)=\left(\sum_{i\in S}i\right)\cdot {}^af.$$
Moreover, for any $i\in I\setminus S$ we have $i\cdot {}^af={}^af\cdot i=0.$ Clearly, if either $e$ or $f$ is primitive in $A$ then $|S|=1.$ Since $f\in Z(A)$ the rest of the proof is immediate.

\end{proof}
In this section we let $A$ denote a $p$-permutation $G$-algebra and $c$ denote a primitive idempotent of the subalgebra $A^G$ of $A.$  
We recall that
$(Q,e)$ is an $(A,c,G)$- Brauer pair if $Q$ is a $p$-subgroup of $G$ and $e$ is a primitive idempotent  of $Z(A(Q))$ satisfying  $\Br_Q^A(c)e\neq 0.$ The set of $(A,c,G)$-Brauer pairs is a $G$-poset with the order relation denoted as usual "$\leq$", see \cite[Part IV, Definition 2.9]{AsKeOl}.
The definition of a \textit{saturated triple}  $(A,c,G),$ that was first introduced in  \cite[Introduction p.92]{KeKuMi}, requires  $c$ to be a \textit{central} idempotent of $A$ and, for any $(A,c,G)$-Brauer pair $(Q,e),$ the idempotent $e$ to \textit{be primitive} in $A(Q)^{C_G(Q,e)}.$ The reason for  calling $(A,c,G)$ a saturated  triple is a consequence of \cite[Theorem 1.6]{KeKuMi}. There, the authors show that if, $(A,c,G)$ is a triple which satisfy conditions (i) and (ii) of \cite[Theorem 1.6]{KeKuMi}, then the poset of subpairs, of any maximal  $(A,c,G)$-Brauer pair,  is a \textit{saturated} fusion system. Verifying the proof of \cite[Theorem 1.6]{KeKuMi}, we noticed that the condition (i) of this theorem, the fact that $c$ is a central idempotent of $A,$ is needed only in \cite[Lemma 3.2]{KeKuMi}.   However, by   Proposition \ref{prop22} below we can  now state that the conclusion of  \cite[Lemma 3.2]{KeKuMi} still hold by replacing condition (i) of  \cite[Theorem 1.6]{KeKuMi} with condition (ii) of  \cite[Theorem 1.6]{KeKuMi}. 
\begin{prop}\label{prop22} Let $A$ be a $p$-permutation $G$-algebra and $c$ be a primitive idempotent of $A^G$. Let $Q$ be a $p$-subgroup of $G$ and $e$ a block of $A(Q)$. If $e$ is a primitive idempotent of $(A(Q))^{C_G(Q,e)}$ then $(Q,e)$ is a $(A,c,G)$-Brauer pair if and only if $\Br_Q^A(c)e=e.$
\end{prop}

\begin{proof}
If $(Q,e)$ is an $(A,c,G)$-Brauer pair then $\Br_Q^A(c)e\neq 0.$ The $N_G(Q)$-algebra surjective homomorphism $\Br_Q^A:A^Q\to A(Q)$ gives $\Br_Q^A(A^G)\subseteq A(Q)^{N_G(Q)},$ hence $\Br_Q^A(c)\in (A(Q))^{C_G(Q,e)}.$ At this point Lemma \ref{Lemma71}, a) establishes the result.
\end{proof}
In conclusion, in our opinion, we can \textit{reformulate} the definition of a saturated triple by eliminating condition (i) of  \cite[Theorem 1.6]{KeKuMi}. 
The above proposition  is also necessary   for other results in the sequel of this paper.

\section{Generalized Brauer pairs and $p$-permutation $N$-interior $G$-algebras}\label{sec3}
In this section  $A$ is an  $N$-interior $p$-permutation $G$-algebra, where $N\trianglelefteq G$ and  $c$ is a primitive idempotent of $A^G.$

Note  that, if $(Q,f_Q)$ is an $(A,c,G)$-Brauer pair, since $f_Q$ is a block idempotent of $A(Q)$ and $A(Q)$ in a $C_N(Q)$-interior algebra, we have $f_Q\in A(Q)^{C_N(Q)}$ and    that $\Br_Q^A(c)f_Q$ is an idempotent of $A(Q)^{C_N(Q)}.$ 
The next proposition is an extension of of \cite[Lemma 8.6.4]{Link} to the case of   $N$-interior  $p$-permutation $G$-algebras. Although the proof is similar to that of \cite[Lemma 3.4]{KeKuMi}, for completeness we present it here.
\begin{prop}\label{prop23}
Let $(Q,f_Q)$ denote an $(A,c,G)$-Brauer pair verifying that $f_Q$ is primitive in $A(Q)^{C_N(Q)}$. Consider a subgroup $H$ of $G$ with  $C_N(Q)\leq H\leq N_G(Q, f_Q)$,  and a $p$-subgroup $S$ of $G$ such that $Q\leq S\leq H.$ Also let $f_S$ be a block of $A(S)$. The following are equivalent:
\begin{itemize}
    \item[i)] $(S,f_S)$ is an $(A,c,G)$- Brauer pair with $(Q,f_Q)\leq(S,f_S)$;
    \item[ii)] $(S,f_S)$ is an $(A(Q),f_Q,H)$-Brauer pair.
\end{itemize}

\end{prop}
\begin{proof}
First we prove "$i)\Rightarrow ii)$". Since $A(Q)^H\subseteq A(Q)^{C_N(Q)}$ it follows that $f_Q$ remains a primitive idempotent in $A(Q)^H$. In \cite[Theorem 2.1]{KeKuMi} it is stated that $(Q,f_Q)$ is the unique Brauer pair included in $(S,f_S)$ and that $f_Q$ is invariant with respect to the $S$-action. That is $\Br_S^{A(Q)}(f_Q)f_S=f_S$.

Conversely for "$ii)\Rightarrow i)$", by our assumptions we have that $\Br_Q^A (c)f_Q=f_Q$ and that $\Br_S (f_Q) f_S =f_S$, see Proposition \ref{prop22}. Indeed, $\Br_S^{A(Q)}:A(Q)^S\rightarrow A(S)$ is a surjective homomorphism  of $k$-algebras, hence $\Br_S^{A(Q)}(Z(A(Q)^S))\subseteq Z(A(S)).$

The commutativity of the diagram
$$\begin{xy} \xymatrix{ A^S\ar@{->}[rr]^{\Br_Q^A}\ar@{->}[dr]_{\Br_S^A} &&A(Q)^S\ar@{->}[dl]^{\Br_S^{A(Q)}}\\ &A(S)    } \end{xy}$$
gives
$$\Br_S^A(c)f_S=\Br_S^{A(Q)}(\Br_Q^A(c))\cdot \Br_S^{A(Q)}(f_Q)f_S  = \Br_S^{A(Q)}(\Br_Q^A(c)\cdot f_Q)f_S$$$$ = \Br_S^{A(Q)}(f_Q)f_S  = f_S$$
This shows that $(S,f_S)$ is an $(A,c,G)$-Brauer pair.

We apply  \cite[Theorem 2.1]{KeKuMi} again to obtain an $S$-invariant block of $A(Q)$, say $f_Q'$, such that $(Q, f_Q')\leq (S,f_S)$ and $\Br_S^{A(Q)}(f_Q')f_S = f_S$. If $f_Q\neq f_Q'$ then 
$$f_S =\Br_S^{A(Q)}(f_Q)f_S = \Br_S^{A(Q)}(f_Q)\Br_S^{A(Q)}(f_Q')f_S=0,$$ a contradiction.
\end {proof}

\section{Isomorphisms of categories of Brauer pairs} \label{sec4}
In this section we continue to assume that $A$ is a $p$-permutation $N$-interior $G$-algebra, where $N\trianglelefteq G$. Let $c$ be primitive idempotent of $A^N$ such that $c\in A^G$. In particular $c$ remains a primitive idempotent of $A^G$.
For shortness, we say that a $p$-subgroup $Q$ of $N$  is a defect group of $c$ in $N$ if $Q$ is a defect group of  $N_{\{c\}}$, considered as  a pointed group on the $G$-algebra $cAc$, see \cite[Proposition 18.5]{Th}. 
 Similarly we say that a $p$-subgroup $P$ of $G$ is defect group of $c$ in $G$.

For any defect group $Q$ of $c$ in $N$ there is a defect group $P$ of $c$ in $G$ such that $Q=P\cap N$. This claim is well known but, for consistency,  we give the justifying details  in the next lines. Indeed, if $Q$ is a defect group in $N$ of $c$, then $c\in A_Q^N$ and $Q$ is minimal with this property. Similarly, if $R$ is a defect group of $c$ in $G$ then $R$ is minimal with the property $c\in A_R^G$. Given the decomposition:
$$c= \underset{x\in [N\setminus {G/R}]}{\Sum} \Tr_{N\cap~^xR  }^N(~^xa),$$ where $a\in A^R$ (see for details \cite[Proposition 11.4]{Th}),
we obtain $c\in A^N_{N\cap~^xR}$, for some $x\in G$. Since $\Br_R^A(c)\neq 0$ it follows $\Br_{~^xR}^A(c)\neq 0$, hence $\Br_{N\cap~^xR}^A(c)\neq 0$. So, there is $y\in N$ with $~^yQ=N\cap~^xR$. It follows $Q=N\cap~^zR,\ z=y^{-1}x$. We set
$P =~^zR$ and the claim is proved.
\begin{nim}\label{rem32}
For the rest of this section we fix a $(A,c,G)$-Brauer pair $(Q,f_Q)$ such that $Q$ is a defect group of $c$ in $N.$ By  the above discussion we choose a maximal $(A,c,G)$-Brauer pair $(P,f_P)$ such that
$(Q,f_Q)\leq (P,f_P)$. We assume that  $f_Q$ remains a primitive  idempotent of $A(Q)^{C_N(Q)}$ and denote by $H$ the group $N_G(Q,f_Q)$.
\end{nim}
\begin{defn}\label{definition33}  Define the sets:
$$F_1:=\{(R,f_R)|(R,f_R)\text{ is a }(A,c,G)\text{-Brauer pair such that}$$
$$(Q,f_Q)\leq(R,f_R)\text{ as }(A,c,G)\text{-Brauer pairs }\}$$
$$F_2:=\{(R,f_R)|(R,f_R)\text{ is a }(A(Q),f_Q,H)\text{-Brauer pair such }$$
$$\text{that }(Q,f_Q)\leq(R,f_R)\text{ as }(A(Q),f_Q, H)\text{-Brauer pairs}\}$$

\end{defn}
\begin{cor}\label{corollary34}
 Let $S$ be a $p$-subgroup of $G$ such that $Q\leq S$ and let $f_S$ be a block of $A(S)$. The following are equivalent:
\begin{itemize}
    \item[a)] $(S,f_S)\in F_1;$
    \item[b)] $(S,f_S)\in F_2.$
\end{itemize}
\end{cor}
\begin{proof}
We first show that $Q\trianglelefteq S$. If b) holds then $S\leq H$ and since $H=N_G(Q,f_Q)$ it follows that $Q$ is normal in $S$. If a) holds, there is $x\in G$ such that
$$(Q,f_Q)\leq (S,f_S)\leq ~^x(P,f_P),$$
hence $Q\leq S\leq ~^xP$. We have $Q=P\cap N,$ thus
$Q\leq~^xP\cap N=~^xQ.$
It follows that $x\in N_G(Q)$ and then  $Q\leq ~^{x^{-1}}S\leq P$. Since $Q$ is normal in $P$ it follows that $~^{x^{-1}}Q$ is normal in $S$, thus $Q\trianglelefteq S$. 

The rest of the proof follows from Proposition \ref{prop23}.
\end{proof}
\begin{prop}
\label{Proposition35}
Let $S,T$ be two $p$-subgroups of $G$ such that $S\leq T$ and let $f_S,f_T$ denote blocks of $A(S),$ $A(T)$, respectively.

The following are equivalent:
\begin{itemize}
    \item[1)] $(S,f_S)\in F_1$ such that $(S,f_S)\leq(T,f_T)$;
    \item[2)] $(S,f_S)\in F_2$ such that $(S,f_S)\leq(T,f_T)$.
\end{itemize}
\end{prop}
\begin{proof} According to Corollary \ref{corollary34}, any pair $(S,f_S)$ is in $F_1$ in and only if it is in $F_2,$ hence we only need to prove  that $(S,f_S)\leq(T,f_T)$ in $F_1$ if and only if $(S,f_S)\leq(T,f_T)$ in $F_2.$ 
Let $e$ denote a primitive idempotent  appearing in a pairwise orthogonal  decomposition of $f_T$ into primitive idempotents in $A(T).$ By lifting idempotents, there exists a primitive idempotent $i\in A^T$  such that $e=\Br_T^A(i)$, hence $\Br_T^A(i)f_T=\Br_T^A(i)$. By our assumption we also have $\Br_S^A(i)f_S\neq 0$. Since $i$ is  primitive in $A^T$, we obtain that $j:=\Br_Q^A(i)$ is also a primitive idempotent of $A(Q)^T$. We have
$$0\neq \Br_T^A(i)f_T=\Br_T^{A(Q)}(j)f_T.$$ 
Similarly  we obtain $$0\neq \Br_S^A(i)f_S=\Br_S^{A(Q)}(\Br_Q^A(i))f_S.$$
This proves that $(S,f_S)\leq (T,f_T)$ as elements of $F_2,$ according to \cite[Corollary 1.9]{BrPu}.

Conversely, considering a similar decomposition of $f_T$ in $A(T)$, we find a primitive idempotent $e$ with $f_Te=f_T$. Let $j$ be such that $\Br_T^{A(Q)}(j)=e$ and let $i$ be such that $\Br_T(i)=e$. Then we may assume $i$ and $j$ to be primitive idempotents in $A^T,A(Q)^T$, respectively. Further we get $\Br_Q^A(i)=~^aj,$ for some $a\in(A(Q)^T)^*.$
We obtain:
$$\Br_T^A(i)f_T = \Br_T^{A(Q)}(~^aj)f_T = f_T\neq 0.$$
Again $\Br_T^A(i)\neq 0$ implies $\Br_S^A(i)\neq 0$ and then
$$0\neq \Br_S^{A(Q)}(~^aj)f_S=\Br_S^A(i)f_S.$$
This concludes the proof.
\end{proof}
\begin{cor}
\label{corollary36} Let  $(P,f_P)$ be an $(A,c,G)$-Brauer pair such that $(Q,f_Q)\leq (P,f_P)$. Then $(P,f_P)$ is a maximal $(A,c,G)$-Brauer pair if and only if it is a maximal $(A(Q),f_Q,H)$-Brauer pair, where $H=N_G(Q,f_Q).$
\end{cor}

Let $(U,f_U)$ be a maximal $(A,c,G)$-Brauer pair. Then $\mathcal{F}_{(U,f_U)}(A,c,G)$ is a finite category, called fusion system,  introduced in \cite[Definition 2.3]{KeKuMi}. It has as objects the $(A,c,G)$-Brauer pairs $(R,f_R)$ such that  $(R,f_R)\leq (U,f_U)$ and the morphisms are described by
$$\mathrm{Hom}_{\mathcal{F}_{(U,f_U)}(A,c,G)}(S,T)=\{c_x: S\rightarrow T| x\in G,\  ~^x(S,f_S)\leq (T,f_T) \},$$
with $c_x(u)=xux^{-1}$ for any $u\in S$ and $S,T\leq U$. Similarly, we introduce the fusion system $\mathcal{F}_{(U,f_U)}(A(Q),f_Q,H).$ 

\begin{prop}
\label{proposition38}
Let  $(P,f_P)$ be a maximal  $(A,c,G)$-Brauer pair such that $(Q,f_Q)\leq (P,f_P)$. If $(S,f_S)\in F_1$ and $(T,f_T)$ is a Brauer pair belonging to $\mathcal{F}_{(P,f_P)}(A,c,G)$ then 
$$\mathrm{Hom}_{\mathcal{F}_{(P,f_P)}(A,c,G)}(S,T)= \mathrm{Hom}_{\mathcal{F}_{(P,f_P)}(A(Q), f_Q,H)}(S,T).$$
\end{prop}
\begin{proof}
Similar to the proof of \cite[Proposition 5]{Ha}.
\end{proof}
We define two subcategories  of the above fusion systems and, in the first main result of this paper, we show that these two categories are isomorphic.

\begin{defn}\label{defnCD}

 The full subcategory of $\mathcal{F}_{(P,f_P)}(A,c,G)$  whose objects are those of $F_1$ is denoted by $\mathcal{C}$. 
 
 The full subcategory of $\mathcal{F}_{(P,f_P)}(A(Q),f_Q,H)$  whose objects are those of $F_2$ is is denoted by $\mathcal{D}$.
 \end{defn}
The proof of the first main result is immediate by collecting the results of this section, applied to the categories given in Definition \ref{defnCD}. 
\begin{thm}
\label{theorem39} Let $A$ be an  $N$-interior $p$-permutation $G$-algebra, where $N\trianglelefteq G$ and let $c$ be primitive idempotent of $A^N$ such that $c\in A^G$. Let $(Q,f_Q)$ be an $(A,c,G)$-Brauer pair such that $Q$ is a defect group of $c$ in $N.$ We fix  a maximal $(A,c,G)$-Brauer pair $(P,f_P)$ such that
$(Q,f_Q)\leq (P,f_P)$ and we assume that  $f_Q$ is  primitive  in  $A(Q)^{C_N(Q)}.$  Then,  the categories $\mathcal{D}$ and $\mathcal{C}$ are isomorphic. 
\end{thm}

\section{Covering points revisited} \label{sec5}


As in Section \ref{sec4},  $N$ is a normal subgroup of the finite group $G.$ We consider  $A,$ a $p$-permutation  $G$-algebra, and we assume that  $A$ contains  a 
$G$-invariant $k$-subalgebra $C$ such that $1_A\in C.$ 
 We recall the definition of covering points that was introduced in \cite{Co}.
 
\begin{defn}\label{defn51}\cite[Definition 3.2]{Co}
\label{definition41} Let $\alpha\subseteq A^G$ be a point and let $\beta\subseteq C^N$ be a point with  defect group $Q.$  We say that the point $\alpha$ \textit{covers} the point $\beta,$ if the following conditions are satisfied:
\begin{itemize}
    \item[$a)$] the point $\alpha$ admits a defect group $D$ such that $D\cap N=Q$;
    \item[$b)$] for any $i\in\alpha$ there is an idempotent $j_1$ of $A^N,$ such that $j_1=~^aj$ for some $j\in \beta, a\in(A^N)^*,$ 
    and there is a primitive idempotent $f\in A^N$ belonging to a point  with defect group $Q$ such that $j_1f=fj_1=f$ and $if=fi=f$.
\end{itemize}
\end{defn}

The following are useful  characterizations of the above definition.
\begin{prop}\label{prop53} The point 
$\alpha\subseteq A^G$ covers the point 
$\beta\subseteq C^N,$ ($Q$ is a defect  group of $\beta$) if and only if for any idempotents $i\in\alpha, j\in \beta$ the following conditions are satisfied:
\begin{itemize}
\item[$a')$] the idempotent $i$ admits a defect group $D$ such that $D\cap N=Q$;
\item[$b')$] there is a primitive idempotent $f_1\in A^N$ with defect group $Q$ such that  
$$i\cdot ~^af_1=~^af_1\cdot i=~^a f_1\text{ and }jf_1=f_1j=f_1\text{ for some } a\in(A^N)^*.$$
\end{itemize}
\end{prop}

\begin{proof}
Assume that $\alpha$ covers $\beta$. Let $i\in\alpha, j\in \beta$. Then condition $a')$ is verified. For the chosen $i\in\alpha, j\in \beta$ there is $j_1=~^{a}j\in\beta, a_1\in(A^N)^*$ such that $j_1f=fj_1=f$ and $if=fi=f$ for some $f\in \delta$, where $\delta\subseteq A^N$ is a point of defect $Q$. We obtain
$$j\cdot ~^{a^{-1}}f=~^{a^{-1}}f\cdot j=~^{a^{-1}}f,\quad if=fi=f,$$ hence, if we denote by $f_1$ the primitive idempotent $~^{a^{-1}}f$ of $A^N$  we obtain
$$jf_1=f_1j=f,\quad i\cdot ~^af_1=~^af_1\cdot i=~^af_1. $$

Assume $a')$ and $b')$ are satisfied for some $i\in\alpha, j\in \beta,$ hence condition $a)$ of Definition \ref{definition41} is verified. By $b')$ there is a point $\delta\subseteq A^N$ with defect $Q$ and $f_1\in\delta$ such that
$$i\cdot ~^af_1=~^af_1\cdot i=~^af_1,\quad jf_1=f_1j=f_1,$$ 
for some  $a\in(A^N)^*.$ But that means
$ {~^aj}\cdot {~^af_1}={~^af_1}\cdot {~^aj}=f_1.$ We denote by $f$ the idempotent $~^af_1\in \delta.$  

Note that $~^aj$ is in the  $(A^N)^*$-conjugacy class  of $\beta$ so there is $j_1,$  a primitive idempotent of $A^N,$ such that 
$ j_1f=fj_1=f.$ 
\end{proof}

We introduce the next definition of covering between points, which is a particular case  of Definition \ref{defn51}. Note that this definition is also equivalent to the usual cover in case of blocks of group algebras that cover a $G$-invariant block of $kN.$
\begin{defn}\label{defn53} Let $\alpha$ denote a point of $A^G$ and  $\beta\subseteq C^N$ be a point with  defect group $Q$ such that $\beta\cap C^G\neq \emptyset.$ We say that $\alpha$ \textit{strongly covers} $\beta$ if,  for any $j\in\beta\cap C^G$ there is $i\in\alpha$ and there is a point $\epsilon\subseteq A^N$ with defect group $Q$ such that:
\begin{itemize}
\item[$a'')$] $ji=ij=i$;
\item[$b'')$] $if=fi=f,\quad jf_1=f_1j=f_1$ for some $f,f_1\in\epsilon$.
\end{itemize} 
\end{defn}

\begin{prop}\label{prop55'} Let $\alpha$ denote a point of $A^G$ and  $\beta\subseteq C^N$ be a point with  defect group $Q$ such that $\beta\cap C^G\neq \emptyset.$  The following statements are true:
\begin{itemize}
\item[(i)] if $\alpha$  strongly covers $\beta$ then $\alpha$ covers $\beta$; 
\item[(ii)] $\alpha$  strongly covers $\beta$ if and only if for any $j\in\beta\cap C^G$ there is $i\in\alpha$ such that $i\in (jAj)^G$ and there is a point $\epsilon\subseteq (jAj)^N$ with  defect group $Q$ such that  $if=fi=f,$ for some $f\in \epsilon$.
\end{itemize}
\end{prop}
\begin{proof}




\begin{itemize}
\item[(i)]

We assume that $\alpha$ strongly covers $\beta$ and let $i_1\in\alpha,j_1\in\beta$. By assumption we can choose $j\in\beta\cap C^G$  and $i\in \alpha$ such  that statements $a'')$ and $b'')$ of Definition \ref{defn53} hold for some point $\epsilon\subseteq A^N$ with defect group $Q.$ Thus, there exist $f,f_1\in\epsilon$  such that
$$if=fi=f,\quad jf_1=f_1j=f_1.$$
Since there is $a_1\in(A^G)^*$ (which is included in $(A^N)^*$) and there is $b_1\in(C^N)^*$ (which is included in $(A^N)^*$) such that $i_1=~^{a_1}i,j_1=~^{b_1}j$, we obtain
$$i_1(~^{a_1}f)=(~^{a_1}f)i_1=~^{a_1}f,\quad j_1(~^{b_1}f_1)=(~^{b_1}f_1)j_1=f_1.$$
We claim that $f':=~^{b_1}f_1$ is the primitive idempotent (with defect group $Q$) which satisfies  condition $b')$ of Proposition \ref{prop53}. Let $a_2\in (A^N)^*$ such that $f=~^{a_2}f_1.$ The claim   is true because
$$~^{a_1a_2b_1^{-1}}f'=~^{a_1a_2}f_1=~^{a_1}f\quad \mbox{ and   \quad  }      j_1f'=f'j_1=f',$$
where $a_1a_2b_1^{-1}\in(A^N)^*.$

Since clearly $ij=ji=i$, we  show next that there is a defect group $D$ of $i$ (hence of $\alpha$) which satisfies $D\cap N=Q,$ i.e. we show that condition $a')$ of Proposition \ref{prop53} is verified. Let $R$ be a defect group of $\alpha$. There is $a\in A^R$ such that
$$i=Tr_R^G(a)=\underset{x\in [N\setminus G/R]}{\sum}Tr_{N\cap ~^xR}^N(~^xa)\in\underset{x\in[N\setminus G/R]}{\sum}A_{N\cap~^xR}^N$$
Since $if=fi=f$ the idempotent $f$ belongs to some ideal $A_{N\cap~^xR}^N$, implying $~^yQ\leq N\cap~^xR$ for some $y\in N$. Then $Q\leq N\cap ~^gR$ for some $g\in G$. Let $T$ denote a defect group of $j$ regarded as an idempotent of $A^G.$ Then $j\in C_T^G\subseteq A_T^G$. We also have $ij=ji=i$ in $A^G$, hence  we may assume $R\leq T$. Next, there is $b\in C^T$ with
$$j=Tr_T^G(b)=\underset{y\in [N\setminus G/T]}{\sum}Tr_{N\cap ~^yT}^N(~^yb)\in\underset{y\in[N\setminus G/T]}{\sum}C_{N\cap ~^yT}^N$$
By our assumption $j\in C_{N\cap ~^yT}^N$ for some $y\in G$. Since $\Br_T(j)\neq 0$ we get $\Br_{~^yT}(j)\neq 0$ and then $\Br_{N\cap ~^yT}(j)\neq 0 $. This means $~^tQ=N\cap ~^yT$ for some $t\in N$, equivalently $Q = N\cap ~^zT$ for some $z\in G$. Finally $$Q\leq N\cap ~^gR\leq N\cap ~^gT = ~^{gz^{-1}}Q,$$ forcing $gz^{-1}\in N_G(Q)$ and $$Q=N\cap ~^gR = N\cap ~^gT.$$
\item[(ii)]

First we assume that $\alpha$ strongly covers $\beta$ and let $j\in \beta\cap C^G$. Then, by $a'')$ of Definition \ref{defn53}, there is $i\in\alpha$ such that $ij=ji=i$, which is equivalent to $i\in (jAj)^G$. Using Definition \ref{defn53}, $b'')$ there is a point $\epsilon'\subseteq A^N$ with defect group $Q$ such that for some $f,f_1\in\epsilon'$ we have
$$jf_1=f_1j=f_1\quad if=fi=f.$$
It follows
$$jf=j(if)=(ji)f=if=(fi)j=fj=f,$$
which is equivalent to the fact that $f\in(jAj)^N$ remains a primitive idempotent in $(jAj)^N$. 
Similarly $jf_1=f_1j=f_1$ is equivalent to the fact that $f_1$ is, a primitive idempotent, in $(jAj)^N.$ Let $\epsilon\subseteq (jAj)^N$ be the point corresponding to $f$ (i.e. 
$\epsilon:=\{afa^{-1}|a\in ((jAj)^N)^*\}$) and let $$F:(jAj)^N\rightarrow A^N$$ be the embedding determined  by the inclusion. Then $\epsilon'$ is the unique point of $A^N$ corresponding to  $\epsilon$ through $F.$ Since $f,f_1\in\epsilon'$ and $f\in \epsilon$ it follows that $f_1\in\epsilon$. Clearly $\epsilon$ has the same defect $Q$ as $f$.


For the converse implication, fix $j\in \beta\cap C^G$. Then there is $i\in\alpha$ such that $i\in (jAj)^G$ and there is a point $\epsilon\subseteq (jAj)^N$ with defect group $Q$ satisfying $if=fi=f$ for some $f\in \epsilon$. 
Then statement $a'')$ of Definition \ref{defn53} is verified. Let $F$ be the same embedding as above and set $\epsilon'=F(\epsilon)$. Then $\epsilon'$ is a point of $A^N$ with defect group $Q$ and  $f\in\epsilon'$ satisfying $jf=fj=f$.
\end{itemize}

\end{proof}

\section{On pairs determined by covering points} \label{sec6}
In this section  we assume that $N$ is a normal subgroup of $G,$ that  $A$ is a $p$-permutation  $G$-algebra  which contains $C$ and $A$ is also an $N$-interior $G$-subalgebra such that $1_A
\in C.$
Let $c$ be  primitive idempotent of $C^N$ such that $c\in C^G$ and let $\beta\subseteq C^N$ be the point containing $c.$ We fix $Q,$ a defect group of $N_{\beta}$. Since $C$ remains a $p$-permutation $G$-algebra, see \cite[Corollary 27.2]{Th},  we choose a maximal $(C,c,N)$-Brauer pair $(Q,f_Q)$ such that $f_Q$ remains  primitive in $C(Q)^{C_N(Q)}.$ It follows that we are in the setup of Section 3 and of Section 4, hence we can introduce $F_1$ and $F_2$ and the conclusion of Theorem \ref{theorem39} holds for $(C,c,G).$

\begin{prop}
\label{proposition53}
Let $\alpha\subseteq A^G$ be a point that strongly covers the point $\beta$. 
Then there is a defect group $D$ of $\alpha$ such that $D\cap N=Q$ and a block $f_D$ in $C(D)$ such that $(D,f_D)\in F_1$.
\end{prop}
\begin{proof}
By Definition  \ref{definition41} we have a defect group $R\leq G$ of $\alpha$ with $R\cap N=Q$. We show next that 
\begin{equation}\label{eq161}
\Br_Q^C(c)f_Q=f_Q=f_Q\Br_Q^C(c) \ \ \text{in}\ C(Q)^{C_N(Q)}.
\end{equation}
We have the inclusion
$C(Q)^{N_G(Q)}\subseteq C(Q)^{C_N(Q)},$
hence since $\Br_Q(c)\in C(Q)^{N_G(Q)}$ it follows that $\Br_Q(c)$ is an idempotent of 
$C(Q)^{C_N(Q)}$ satisfying $\Br_Q(c)f_Q\neq 0$. Since $f_Q$ is primitive in $C(Q)^{C_N(Q)}$ we apply Lemma \ref{Lemma71} to obtain (\ref{eq161}).

For any $x\in N_G(Q)$,  (\ref{eq161}) yields  
\begin{equation}\label{eq261}
\Br_Q^C(c)\cdot ~^xf_Q=~^xf_Q\cdot \Br_Q^C(c)=~^xf_Q,
\end{equation} an equation   viewed in  $C(Q)$. If we set $H=N_G(Q,f_Q)$ then $s:=\underset{x\in [N_G(Q)/H]}{\sum}~^xf_Q$ is an idempotent lying in $C(Q)^{N_G(Q)}\cap Z(C(Q))$   such that $$\Br_Q^C(c)=\Br_Q^C(c)s=s\  \Br_Q^C(c)=s$$ in $C(Q)^{N_G(Q)}.$ The  equality on the right follows from (\ref{eq261}). The  equality on the left  holds since $\Br_Q^C(c)$ is a primitive idempotent of $C(Q)^{N_G(Q)}$ and by  applying  Lemma \ref{Lemma71}. In particular $s=\Br_Q^C(c)$ lies in $$Z(C(Q))\cap C(Q)^R\subseteq Z(C(Q)^R).$$ 

 Since $\alpha$ strongly covers $\beta,$ by Definition  \ref{defn53} there is $i\in\alpha$ with $ic=i=ci$. It follows that $\Br_R^C(c)\neq 0$. The idempotent $\Br_R^C(c)$ belongs to $Z(C(R))$, since  $\Br_Q^C(c)$ is in $Z(Q)$ and $\Br_R^C=\Br_R^{C(Q)}\circ \Br_Q^C|_{C^R}$ is a composition of  surjective homomorphimsms of $k$-algebras. Let $\sum_{f\in \mathfrak{B}} f$ be a  decomposition of $\Br_R^C(c)$ into a sum of blocks, where $\mathfrak{B}$ is a set of block idempotents of  $Z(C(R)).$ Consider  $c=\underset{j\in J}{\sum}j,$  a decomposition of $c$ into  a sum of pairwise orthogonal primitive idempotents of $C^R$. There is $j\in J$ that satisfies the relations    $\Br_R^C(j)\neq 0$ and $cj=jc=j$.  The idempotent 
 $$\Br_R^C(j)=\Br_R^C(cj)=\sum_{f\in \mathfrak{B}} f\Br_R(j)$$ is primitive in $C(R)$ and there is a unique block   $\overline{f_R}\in \mathfrak{B}$ such that $$\Br_R(j)= \Br_R^C(j)\overline{f_R}=\overline{f_R}\Br_R(c).$$
Since $\Br_R^C(j)\neq 0$ then  $\Br_Q^C(j)\neq 0$ and since $j=jc$ we also have $$\Br_Q^C(j)\Br_Q^C(c)=\Br_Q^C(j)s\neq 0.$$ This last relation determines at least one block, say $~^xf_Q$ (which appears in  $s$), with $\Br_Q^C(j)~^xf_Q\neq 0$. According to \cite[Corollary 1.9]{BrPu} the blocks $\overline{f_R}$ of $C(R)$ and $~^xf_Q$ of $C(Q)$ must verify  $(Q,~^xf_Q)\leq (R,\overline{f_R})$. Because $x\in N_G(Q),$ this   becomes $(Q,f_Q)\leq (~^{x^{-1}}R,~^{x^{-1}}\overline{f_R})$. We set $D:=~^{x^{-1}}R$ and $f_D:=~^{x^{-1}}\overline{f_R}$ to conclude the proof.
\end{proof}

\section{Isomorphism of fusion subcategories} \label{sec7}
Throughout this section we use the following setup.
\begin{set}\label{set} We consider     a $G$-interior $k$-algebra $A,$ with  structural map 
$G\rightarrow A^*$ such that $$G\ni u\mapsto u\cdot 1_A=1_A\cdot u\in A,$$ for any $u\in G.$   This makes  $A$ into a $k[G\times G]$-module, explicitly   $$(u,v)\cdot a=u\cdot a\cdot v^{-1}\in A,$$ for any $u,v\in G,a\in A$. In particular we assume that
this action stabilizes  a $k$-basis of $A,$ so that $A$ is also a $p$-permutation $G$-algebra.  We denote by  $C$ an $N$-interior  $G$-subalgebra of $A$ that contains the identity element $1_A.$ 
Let $c$ denote a  primitive idempotent of $C^N$ such that $c\in C^G$ and let $\beta\subseteq C^N$ be the point containing $c.$  We fix $Q,$ a defect group of $N_{\beta}$. As in Section \ref{sec6}, we consider $(Q,f_Q),$  a maximal $(C,c, N)$-Brauer pair such that  $f_Q$ remains  primitive  in $C(Q)^{C_N(Q)}.$ Further we set $H = N_G(Q,f_Q)$. 

Let $X:=\overline{N}^{\Aut(Q)}_A(Q)$  denote the extended Brauer quotient, the  $N_G(Q)$-interior algebra constructed in \cite{PuZh}. Note that $C(Q)$ is an $N_N(Q)$-interior $N_G(Q)$-subalgebra of $X$ 
and if $A=kG$ then \cite[Remark 4.5, 2)]{CoMa} states $X=kN_G(Q).$

We add the following assumption on $f_Q.$ 
We assume that $f_Q$  \textit{normalizes} a $k$-basis $\mathcal{B}$ of $X^{N_G(Q)}.$  Explicitly  for any $s\in \mathcal{B}$ we have  $f_Q\cdot s=s'\cdot f_Q$ for some $s'\in\mathcal{B}.$ Since $f_Q\in Z(C(Q))$  the point of $C(Q)^{C_N(Q)}$ containing $f_Q$ has a unique element. We  denote this point by $\{f_Q \}.$ We also assume that $A$ is  projective as a $k[G\times 1]$- module and as a $k[1\times G]$-module.
\end{set}

\begin{lem}
\label{lemma61}
Any point of  $X^H$ that strongly covers $\{f_Q\}$ has a defect group $D$ such that $D\cap N = Q$ and, there is a block $f_D$ of $C(D)$ such that $(Q,f_Q)\leq (D,f_D)$ as $(C(Q),f_Q, H)$-Brauer pairs.
\end{lem}
\begin{proof}
Let $\bar{\gamma}\subset X^H$ be a point that strongly covers $\{f_Q\}\subseteq C(Q)^{C_N(Q)}.$ Let $D$ denote a defect group of $\bar{\gamma},$ hence $D\cap QC_N(Q)=Q.$ There is $\bar{j}\in \bar{\gamma}$ such that 
$$\bar{j}\in (f_QXf_Q)^H_{D}\subseteq \sum_{x\in [N_N(Q,f_Q)\setminus H/D]}(f_QXf_Q)^{N_N(Q,f_Q)}_{N_N(Q,f_Q)\cap ^{x}D}.$$
The block $f_Q$ of $C(Q)$ remains a primitive idempotent of $C(Q)^L$, for any subgroup $L$ of $H$ containing $QC_N(Q)$. In particular $\{f_Q\}$ is a point of $C(Q)^{N_N(Q,f_Q)}$ with  defect group $Q.$  
Any decomposition of $\bar{j}$  in $(f_QXf_Q)^{N_N(Q,f_Q)}$ contains at least one idempotent that belongs to a point, say $\overline{\epsilon}\subseteq (f_QXf_Q)^{N_N(Q,f_Q)},$ with defect group $N_N(Q,f_Q)\cap ^{x}D,$ for some $x\in H.$ 
Since  $\overline{\epsilon}$ is a subset of  $(f_QXf_Q)^{N_N(Q,f_Q)}_Q$ and, since
  $Q\leq N_N(Q,f_Q)\cap ^{x}D$ we must have 
  $D\cap N_N(Q,f_Q)=Q.$ It follows that $\overline{\gamma}$ strongly covers $\{f_Q\}\subseteq C(Q)^{N_N(Q,f_Q)}$ and this forces 
  $$Q=D\cap N_N(Q,f_Q)=D\cap N_N(Q)=D\cap N.$$   

We apply Proposition \ref{proposition53} with $H,QC_N(Q), X,$ $C(Q)$ and $\{f_Q\}$, in place of $G,N,A,C$ and $\beta$ respectively. 
 
\end{proof}

\begin{lem}\label{lem73'} The following statements hold:
\begin{itemize}
\item[a)] The $k$-algebras $\Br_Q^C(c)X^{N_G(Q)}\Br_Q^C(c)$ and $(f_QXf_Q)^H$ are isomorphic;
\item[b)] The above isomorphism determines a bijective correspondence between the points $\overline{\alpha}\subseteq X^{N_G(Q)}$ such that for some $\overline{i}\in\overline{\alpha}$ we have $\Br_Q^C(c)\cdot \overline{i}=\overline{i}\cdot \Br_Q^C(c)=\overline{i}$ and the points $\overline{\gamma}\subseteq X^H$ such that for some $\overline{j}\in\overline{\gamma}$ we have $\overline{j}\cdot f_Q=f_Q\cdot \overline{j}=\overline{j}.$  This bijection preserves the defect groups in the sense that any defect group of $\bar{\gamma}$  is also a defect group of $\bar{\alpha}$ and conversely, some $N_G(Q)$-conjugate of any   defect group of $\bar{\alpha}$ is a defect group of $\bar{\gamma}.$
\end{itemize}
\end{lem}
\begin{proof}
\begin{itemize}
\item[a)] The map
$$\Tr_H^{N_G(Q)}:(f_QXf_Q)^H\rightarrow(f_QXf_Q)_H^{N_G(Q)}, $$
$$(f_QXf_Q)^H\ni a\mapsto \Tr_H^{N_G(Q)}(a):=\sum_{x\in[N_G(Q)/H]}{}^xa\in (f_QXf_Q)_H^{N_G(Q)},$$

is a $k$-algebra isomorphism. It is easy to verify that an inverse map of $\Tr_H^{N_G(Q)}$ maps $b\in (f_QXf_Q)_H^{N_G(Q)}$ to $f_Qb\in (f_QXf_Q)^H$ and that $\Tr_H^{N_G(Q)}$ is an algebra homomorphism.  Next we show that $(f_QXf_Q)_H^{N_G(Q)}$ is an ideal of $\Br_Q^C(c)X^{N_G(Q)}\Br_Q^C(c)$. Let $a_1\in f_QXf_Q$ and let  $a_2\in \Br_Q^C(c)X^{N_G(Q)}\Br_Q^C(c)$. We have
$$\Tr_H^{N_G(Q)}(a_1)\cdot a_2=\Tr_H^{N_G(Q)}(a_1a_2)=\Tr_H^{N_G(Q)}(f_Qa_1f_Qa_2\Br_Q^C(c))$$
$$=\Tr_H^{N_G(Q)}(f_Qa_1a_2'f_Q\Br_Q^C(c))=\Tr_H^{N_G(Q)}(f_Qa_1a_2'f_Q),$$
where $a_2'$ is an element of $X^{N_G(Q)}$ such that $f_Qa_2=a'_2f_Q.$ 
 Since $$\Br_Q^C(c)=\Tr_H^{N_G(Q)}(f_Q)\in (f_QXf_Q)_H^{N_G(Q)}$$ we obtain $$(f_QXf_Q)_H^{N_G(Q)}=\Br_Q^C(c)X^{N_G(Q)}\Br_Q^C(c).$$

\item[b)] The $N_G(Q)$-interior algebra embedding 
$\Br_Q^C(c)X\Br_Q^C(c)\to X,$ the $H$-interior algebra embedding $f_QXf_Q\to X$ and  $a)$ determine the mentioned bijection. We only need to verify the second part of the assertion.

Consider a point $\overline{\alpha}\subseteq X^{N_G(Q)}$ and 
let $\overline{\gamma}\subseteq X^H$ be the point corresponding to $\overline{\alpha}.$  The mentioned embeddings determine unique points $\overline{\alpha}_1$ and $\overline{\gamma}_1,$ corresponding to $\overline{\alpha}$ and $\overline{\gamma}$ respectively, such that $\Tr_H^{N_G(Q)}(\overline{\gamma}_1)=\overline{\alpha}_1$ and $f_Q\cdot \overline{\alpha}_1=\overline{\gamma}_1.$  

It suffices  to show that $\overline{\alpha_1}$ and $\overline{\gamma_1}$ verify the statement on defect groups. Let $D$ be a defect group of $\overline{\gamma_1},$ then   $\overline{\alpha_1}\subseteq X_D^{N_G(Q)}$. Assuming by contradiction that $\overline{\alpha_1}\subseteq \sum_{T<D}X_T^D$ we obtain $f_Q\cdot \overline{\alpha_1}\subseteq \sum_{T<D}X_T^D.$ 
If $R\leq N_G(Q)$ is a defect group of $\overline{\alpha}_1$ it is clear that $^{x}R=D$ for some $x\in N_G(Q).$

\end{itemize}
\end{proof}

\begin{lem}\label{lem73''}  The  inclusion  $A(Q)\subseteq X$ of $N_G(Q)$-algebras determines a defect group preserving bijective correspondence between the points of  $ \Br_Q^C(c)A(Q)^{N_G(Q)}\Br_Q^C(c)$  and the points of $\Br_Q^C(c)X^{N_G(Q)}\Br_Q^C(c).$
\end{lem}
\begin{proof} Since $A=cAc\oplus (1-c)A(1-c)$  it follows that $cAc$  is a projective  $k[G\times 1]$-module (and  $k[1\times G]$-module). Also note that $(cAc)(Q)=\Br_Q^C(c)A(Q)\Br_Q^C(c)$ and
$$\overline{N}_{cAc}^{\Aut(Q)}(Q)\cong \Br_Q^C(c)X\Br_Q^C(c),$$
as $N_G(Q)$- algebras and as $N_G(Q)$-interior algebras respectively.
The conclusion now follows  by applying  \cite[Proposition 3.3]{PuZh} to the $G$-interior algebra $cAc.$

\end{proof}

We consider the set
\[\mathfrak{D}=\{D\leq G\mid D \mbox{ is a $p$-subgroup with } D\cap N=Q \}.\]

\begin{thm}\label{thecorresp} There is a  bijective correspondence preserving the defect groups in $\mathfrak{D}$ between the points of  $A^G$ that strongly cover $\beta$ and the points of  $X^H$ that strongly cover $\{f_Q\}$.
\end{thm}

\begin{proof}
Let $\alpha\subseteq A^G$ be a point that strongly covers $\beta.$ Let  $D\in \mathfrak{D}$ denote a defect group of $\alpha.$ 
According to \cite[Theorem 3.5]{Co},  Lemma \ref{lem73''} and Lemma \ref{lem73'} there is a unique point $\overline{\gamma}\subseteq X^H$ that corresponds to $\alpha.$ We obtain that $R:=^{g}D\leq H,$ for some $g\in N_G(Q),$ is a defect group of $\overline{\gamma}$ and that $\bar{j}\in f_QX^Hf_Q,$ for some $\bar{j}\in \overline{\gamma}.$
 We obtain 
$$\bar{j}\in (f_QXf_Q)_R^H\subseteq \sum_{x\in[QC_N(Q)\setminus H/R]}(f_QXf_Q)_{QC_N(Q)\cap{}^x R}^{QC_N(Q)}=(f_QXf_Q)_Q^{QC_N(Q)},$$
 since for all $x\in H$ we have
$$Q\leq QC_N(Q)\cap {}^x R={}^x(QC_N(Q)\cap R)\leq{}^x(N\cap R)={}^xQ=Q.$$ 
It follows that $\overline{j}\in (f_QXf_Q)_Q^{QC_N(Q)}$ and since $\Br_D^A(\overline{j})\neq 0$ we have $\Br_Q^A(\overline{j})\neq 0$. Consider $\bar{j}=\sum \overline{f},$ a primitive decomposition of $\bar{j}$ in $(f_QXf_Q)_Q^{QC_N(Q)}$. There is at least one idempotent in the above decomposition, say $\overline{f}$ such that $\Br_Q^A(\overline{f})\neq 0$. We denote by $\overline{\epsilon}$ the point of $(f_QXf_Q)^{QC_N(Q)}$ that contains $\overline{f}.$
Applying Proposition \ref{prop55'} ii) we obtain that $\overline{\gamma}$ strongly covers $\{f_Q\}.$

Conversely, let $\overline{\gamma}$ be a point of $X^H$ that strongly covers the point $\{f_Q\}\subseteq (C(Q))^{QC_N(Q)}.$ Denote by $D$ defect group of $\overline{\gamma}.$  Lemma \ref{lemma61} states that $D\in \mathfrak{D}.$ 
Proposition \ref{prop55'} ii) determines  $\overline{j}\in \overline{\gamma}$ such that $\overline{j}\in (f_QXf_Q)^H$. Consider the point $\overline{\gamma_1}$ of $(f_QXf_Q)^H$ that contains $\overline{j}$.

By applying Lemma \ref{lem73'} and Lemma \ref{lem73''} there is a unique point $\overline{\alpha}\subseteq A(Q)^{N_G(Q)},$ admitting defect group $D,$ such that $\bar{i}\in \Br^C_Q(c)A(Q)^{N_G(Q)}\Br^C_Q(c)$ for some $\bar{i}\in  \overline{\alpha}.$
Since we have
$$\bar{i}\in \Br_Q^C(c)\cdot A(Q)^{N_G(Q)}\cdot \Br_Q^C(c)\subseteq$$ 
$$ \sum_{x\in [N_N(Q)\setminus N_G(Q)/D]}(\Br_Q^C(c)\cdot A(Q)^{N_G(Q)}\cdot\Br_Q^C(c))_{N_N(Q)\cap {}^xD}^{N_N(Q)}$$
we obtain  $\bar{i}\in (\Br_Q^C(c)\cdot A(Q)\cdot \Br_Q^C(c))_Q^{N_N(Q)}.$ Let $\bar{i}=\sum \overline{f'}$ be a primitive decomposition of $\bar{i}$  in $ (\Br_Q^C(c)\cdot A(Q)\cdot \Br_Q^C(c))_Q^{N_N(Q)}.$ Since $\Br_D^{A(Q)}(\overline{i})\neq 0$ it follows that there  is a primitive idempotent $$\overline{f'}\in (\Br_Q^C(c)\cdot A(Q)\cdot \Br_Q^C(c))_Q^{N_N(Q)}$$  such that $\Br_Q^{A(Q)}(\overline{f'})\neq 0$. We denote by $\overline{\epsilon'}\subseteq (\Br_Q^C(c)\cdot A(Q)\cdot \Br_Q^C(c))^{N_N(Q)}$ the point that contains $\overline{f'}$. We obtained that $\overline{\alpha}$ strongly covers $\overline{\beta}.$ 

\end{proof}

Let $\alpha\subseteq A^G$ and $\overline{\gamma}\subseteq X^H$ be two points lying in the correspondence determined by Theorem \ref{thecorresp}.  Hence $\alpha$ strongly covers $\beta$ and $\overline{\gamma}$ strongly covers $\{f_Q\}.$ Let $D\in\mathfrak{D}$ be a defect group of $\overline{\gamma}$ that verifies the conclusion of Lemma \ref{lemma61}. So, there is a block $f_D\in C(D)$ such that  $(Q,f_Q)\leq (D,f_D)$ as $(C(Q),f_Q,H)$-Brauer pairs and as $(C,c,G)$-Brauer pairs.
We fix a maximal $(C,c,G)$-Brauer pair $(P,f_P)$ such that $(D,f_D)\leq (P,f_P).$ Let $\mathcal{C}$ and $\mathcal{D}$ be the categories defined as in Definition \ref{defnCD}, for the triple $(C,c, G)$. Let $\mathcal{C}(D,f_D)$ denote the full subcategory of $\mathcal{C}$  consisting of $(C,c,G)$-Brauer pairs $(T,f_T)$ such that
$(T,f_T)\leq(D,f_D).$ Let $\mathcal{D}(D,f_D)$ denote the full subcategory of $\mathcal{D}$ consisting of $(C(Q),f_Q,H)$-Brauer pairs $(T,f_T)$ with $(T,f_T)\leq(D,f_D)$.

As a consequence of Theorem \ref{thecorresp} and Theorem \ref{theorem39} we obtain the final main result of our paper.
\begin{thm}\label{theorem63}
The isomorphism between the categories $\mathcal{C}$ and $\mathcal{D}$ determines an isomorphism between  $\mathcal{C}(D,f_D)$ and $\mathcal{D}(D,f_D)$.
\end{thm}

\textbf{Acknowledgments.}
This work was supported by a grant of the Ministry of Research, Innovation and Digitization, CNCS/CCCDI–UEFISCDI, project number PN-III-P1-1.1-TE-2019-0136, within PNCDI III.


\end{document}